\documentclass[twoside]{article}
\author{\normalsize{WEIWEN ZHANG}}
\date{}
\title{\large\textbf{ROTH-TYPE THEOREMS IN ADDTIVE COMBINATORICS}}
\usepackage{bm}
\usepackage{amsmath}
\usepackage{amsthm}
\usepackage{geometry}
\usepackage{amssymb}
\usepackage{graphicx}
\usepackage{subfigure}
\numberwithin{figure}{section}
\usepackage{float}
\usepackage{extarrows}
\usepackage{tikz}
\usepackage{hyperref}
\usepackage{url}
\usepackage{setspace}
\usepackage{threeparttable}
\usepackage[numbers,sort&compress]{natbib}
\hypersetup{
	colorlinks=true,
	linkcolor=red
}
\usepackage{indentfirst}
\RequirePackage{titlesec}
\titleformat{\section}
{\centering\normalsize}{\thesection.}{0.5em}{}
\titleformat{\subsection}{\normalsize\bfseries}{\thesubsection.}{0.5em}{}
\titleformat{\subsubsection}{\normalsize}{\thesubsubsection.}{0.5em}{}
\usepackage{fancyhdr} 

\begin{document}

	\maketitle
	\begin{quotation}
		\textsc{Abstract.}
		In this article we will introduce a central problem in additive combinatorics, which arised from the famous van der Waerden theorem and an early conjecture of Erd\H{o}s and Tur\'{a}n. The first important theorem was due to Roth in 1953. There were a number of generalized or improved  results afterwards, which we call Roth-type theorems. We will list them and try to give concise expositions to the ideas in some of the proofs without much prior knowledge.
	\end{quotation}
	\newtheorem{thm}{\textbf{Theorem}}[section] 
	\newtheorem{cor}[thm]{\textbf{Corollary}}
	\newtheorem{prop}[thm]{\textbf{Proposition}}
	\newtheorem{lmm}[thm]{\textbf{Lemma}}  
	\newtheorem{defi}[thm]{\textbf{Definition}}  
	\newtheorem{clm}[thm]{\textbf{Claim}} 
	\newtheorem{conj}[thm]{Conjecture}
	
	\section{\textsc{Introduction}}\label{1}

	Additive combinatorics has seen a rapid development in recent decades, due to the efforts of many great mathematicians including young ones who worked on the intersection field of combinatorics, number theory, along with ergodic theory and graph theory, which was discovered to be also important and useful in this appearently more "number theoritic" and "combinatorial" field. In the long-standing history of number theory, some of the problems  naturally have combinatorial backgrounds. For example sometimes we need to count certain configurations in a subset of integers. Among them some are linear with respect to addition thus we can transform our problems into sumsets. Counting additive strucutre thus motivated Tao and Vu to coin the word "additive combinatorics" for this topic. Their textbook of the same name\cite{AC} is a standard material, whilst it doesn't cover more recent developments. For external reading we recommend the booklet by Green\cite{ID} and the textbook by Zhao\cite{GTAC}.\par
	In additive combinatorics, the major object we are concerned about is the additive set, i.e., some $A\subset G$, where $(G,+)$ is an Abelian group, known as the ambient group of $A$.  For example, the most frequently used setting is the cyclic group $\mathbb{Z}$ or $\mathbb{Z}/N \mathbb{Z}$, the Euclidean space $\mathbb{R}^{n}$ and the vector space over a finite field $\mathbb{F}_{q}^{n}$.\par
	The most important two perspectives of an additive set $A$ are its size and structure. The size may be regarded as the global information while there are complex structures that reflect the local information of $A$. For example, we may ask how close $A$ is to a subgroup $H\leqslant G$, whether $A$ contains a special pattern such as $\{x,x+d,x+2d\}$ for some $x\in G,d\neq 0$. Naturally problems arise from the transformation between these two perspectives. Or how well exactly can we deduce complex local information from simple global information?\vspace{2ex}\par
	\textit{From size to structure.} Given certain structures of a set, we can bound from above or below its size. Since the aim of Roth-type problems is to give the upper bound, we now here give an example of estimating the lower bound. A Kakeya set in $\mathbb{R}^{n}(n\geqslant 2)$ is a compact subset that contains a unit line segment pointing in every direction. The famous Kakeya conjecture asserts that:
	\begin{conj}
		Any Kakeya set in $\mathbb{R}^{n}(n\geqslant 2)$ has \textit{Minkowski} and \textit{Hausdorff} dimension $n$.
	\end{conj}
	Up to now, only the cases $n=2$(due to Davies\cite{kk2}) and $n=3$(due to Wang-Zahl\cite{kk3}) were solved but higher dimensions remain open. Later we shall mention the corresponding Kakeya set conjecture in finite field setting, which was solved by Dvir\cite{kkf}.
	\vspace{2ex}\par
	\textit{From size to structure.} This direction, known as the inverse problem, has been the focus of additive combinatorics. Here the size information we mostly exploit are the double constant and the additive energy, which we shall define later. There have been celebrated results, for instance, the Balog-Szemer\'{e}di-Gowers theorem(see \cite[Theorem 2.29]{AC}), the Bogolyubov-Ruzsa lemma(see \cite[Theorem 1.1]{BR}), the polynomial Freiman-Rusza conjecture(see \cite[Conjecture 5.34]{AC}), etc. They characterize a set by "inner" or "outer" structures when its double constant is small or its additive energy is large. Here we give a version of the polynomial Freiman-Rusza conjecture in characteristic 2.
	\begin{conj}
		Suppose $A\subset\mathbb{F}_{2}^{n}$ is a set with $|A+A|\leqslant{K|A|}.$ Then $A$ is covered by at most $2K^{C}$ cosets of some subgroup $H\leqslant{\mathbb{F}_{2}^{n}}$ of size at most $|A|$.
	\end{conj}
	Previously we only knew the conjecture was right if we replace the number of cosets with bound worse than the polynomial type, say $2K^{2}2^{K^{4}}$ or $\exp (\log^{4}K)$. In 2023, the team of Gowers-Green-Manners-Tao proved this conjecture with $C=12$\cite{PFR}, which was a tremendous breakthrough.\par
	We shall also remark that the inverse theorems are often strong and by using them we may get better estimates of the size of a set with certain structure. Actually in improving some Roth-type theorems we rely on the Bogolyubov-Ruzsa lemma and several results based on it(see section \ref{4}).
	
	\section{\textsc{Basic Notation and Definitions}}\label{2}
	\subsection{Combinatorial Side}
	Suppose $A\subset G$ is an additive set, we define the \textit{sum set} $A+A:=\{a+b:a,b\in{A}\}$, \textit{difference set} $A-A:=\{a-b:a,b\in{A}\}$; for $k\in{\mathbb{Z^{+}}}: $ \textit{iterated sumset} $kA$ $:=\{a_{1}+\cdots+a_{k}:a_{1},...,a_{k}\in{A}\}$, \textit{dilation} $k\cdot A$ $:=\{ka:a\in{A}\}$. The double constant $d(A):=|A+A|/|A|$.\par
	 For any $N\in{\mathbb{Z^{+}}}$, $[N]:=\{1,2,...,N\}$. An \textit{arithmetic progression} of length $k$(abbr. $k$-AP) is $x+d\cdot([k]-1)=\{x,x+d,\cdots,x+(k-1)d\}$ for some $x,d\in{G}$. We say a $k$-AP is nontrivial if $d\neq{0}.$
	 There is actually an abuse of notation since from only the set of AP we may not be able to tell $x$ and $d$ or even its length. However we don't care about trivial cases and the base point $x$ is not important because we mainly study translation invariant equations in this paper. We can further define the generalized arithmetic progression. For simplicity we leave interested readers to refer to \cite{AC}.
	 
	 \subsection{Analytic Side}
	  \textit{Asymptotic notations}: $X\ll{}Y\iff{}X=O(Y)\iff{}|X|\leqslant{}C\,|Y|$ for some absolute constant $C>0$; $X\ll_{C_{1},...C_{r}}Y$ means $C$ depends on ${C_{1},...C_{r}}$; $X\asymp Y$ if $X\ll Y$ and $Y\ll  X$; \par
 Suppose $A\subset{B}\subset{\mathbb{Z^{+}}},$ the 	 \textit{relative upper density} \[\overline{d}_{B}(A):=\varlimsup\limits_{N\rightarrow{\infty}}{\frac{|A\cap{}[N]|}{|B\cap{}[N]|}}.\]
	 When the limit exists, especailly in the finite setting, we denote instead by $d_{B}(A)$ the relative density. We omit $B$ if it is clear according to the context and just say the density.\par
	 Suppose $G$ is finite, given a probability measure $\mu$ on $G$, which can be regarded as a non-negative real function with $\|\mu\|_{1}=1$(e.g. for $V\subset{G}$, the measure corresponding to $\mu_{V}=\frac{|G|}{|V|}1_{V}$), $f,g:G\to\mathbb{C} $, $1\leqslant{p}<\infty$, we define:\vspace{-2pt}
	 \[\text{\textit{$\mu$-inner product}}\  \langle{f,g}\rangle_{\mu}=\mathop{\mbox{\raisebox{-4pt}{\huge$\mathbb{E}$}}}\limits_{x\in{G}}\mu(x)f(x)\overline{g(x)}=\frac{1}{|G|}\sum\limits_{x\in{G}}\mu(x)f(x)\overline{g(x)}\]
	 \vspace{-8pt}
	 \[L^{p}(d\mu)\text{\textit{-norm} }\|f\|_{p(\mu)}=(\mathop{\mbox{\raisebox{-4pt}{\huge$\mathbb{E}$}}}\limits_{x\in{G}}\mu(x)|f(x)|^{p})^{1/p}
	 \]
	 The dual group of $G$ is Hom$(G,\mathbb{C}^{\times})=\widehat{G}\cong{G}$. The Fourier transformation of $f$ is defined on $\widehat{G}$ as $\widehat{f}(\gamma)=\langle{f,\gamma}\rangle=\mathbb{E}_{x\in{G}}f(x)\overline{\gamma(x)}$, $\forall\gamma\in\widehat{G}$. \textit{Convolution} $f\ast{}g(x)=\mathbb{E}_{y\in{G}}f(y){g(x-y)}$, \textit{difference convolution} $f\circ{}g(x)=\mathbb{E}_{y\in{G}}f(x+y)\overline{g(y)}$.
	  Suppose $G=\mathbb{Z}$, then its dual $\widehat{G}\cong{\mathbb{T}}$. For $f: \mathbb{T}\to\mathbb{C}$, $g: \mathbb{Z}\to\mathbb{C}$, we define their Fourier transformation by\vspace{-4pt} \[\widehat{f}(n)=\int_{\mathbb{T}}f(\theta)e(-n\theta)d\theta, \forall n\in\mathbb{Z};\quad \widehat{g}(\theta)=\sum\limits_{n\in\mathbb{Z}}g(n)e(n\theta), \forall \theta\in\mathbb{T},\setlength\belowdisplayskip{3pt}\]
	 where $e(x):=e^{2\pi{}ix}.$\par
	 For analytic approaches we prefer to work on $\mathbb{T}$ after transference from $\mathbb{Z}$, while for combinatorial approaches we usually need to restrict $\mathbb{Z}$ to $\mathbb{Z}/N\mathbb{Z}$ for large integer $N$ to discretize both the ambient group and its dual due to the nice property that $\widehat{\mathbb{Z}/N\mathbb{Z}}\cong\mathbb{Z}/N\mathbb{Z}$.
	 \par
	 Suppose $G$ is any Abelian group, $f,g,h$ are complex-valued functions defined on $G$, then we have the following facts between physical space and Fourier space.
	 \begin{itemize}
	 \item $\widehat{f\ast{g}}=\widehat{f}\ \widehat{g}$, $\widehat{f\circ{f}}=|\widehat{f}\,|^{2}$, $\langle{f,g\ast{h}\rangle}=\langle{f\circ{g},h}\rangle$;
	 \item (Fourier inverse formula, finite case)$f(x)=\sum_{\gamma\in\widehat{G}}\widehat{f}(\gamma)\gamma(x)$;
	 \item(Parseval's identity)$\langle{f,g}\rangle=\langle{\widehat{f},\widehat{g}}\rangle$, especially when $g=f$, we have $\|f\|_{L^{2}}=\|\widehat{f}\|_{L^{2}}.$
	 
	 \end{itemize}
	\section{\textsc{Overview of Roth-Type Theorems}}
	In this section we shall give the motivation and history of our topic, focusing on the demonstration of the results and how early researches inspired the newer ones. We will try not to include their methods in this section in case of digression and leave this to section \ref{4}.
\subsection{Early Motivation}
In combinatorics there is an interesting and famous theorem due to van der Waerden:
\begin{thm}[van der Waerden, 1927]\label{3.1}
	$\forall{h,k\in\mathbb{Z^{+}}}, \exists{N(h,k)}$ s.t. $\forall N\geqslant{N(h,k)}, [N]=C_{1}\cup\cdots\cup{C_{h}}, \exists i$ s.t. $C_{i}$ contains a non-trivial $k$\text{-AP}.
\end{thm}
We can establish an equivalence easily between the theorem above and the following version by compactification.

\begin{thm}\label{3.2}
	$\forall{h\in\mathbb{Z^{+}}}, \mathbb{Z^{+}}$(or $\mathbb{Z}$)$=C_{1}\cup\cdots\cup{}C_{h}, $ we can find a non-trivial AP of arbitrary length in some $C_{i}$.
\end{thm}
We note that in Theorem \ref{3.2} $\sup\limits_{1\leqslant{i}\leqslant{h}}\overline{d}(C_{i})>0$, thus we naturally ask:\\
\textbf{Question.} If $A\subset\mathbb{Z^{+}}$ satisfies $\overline{d}(A)>0$, is it true that $A$ must contain arbitrarily long AP?\par
If we can prove so, Theorem \ref{3.1} and \ref{3.2} follow immediately. As we said in section \ref{2} a prior step to carry out is to restrict our set on the finite setting $\mathbb{Z}/N\mathbb{Z}$. This is permissible because we allow that the density differs by a positive constant. For instance, if we want to study 3-AP we are dealing with a simple linear equation
\begin{equation}\label{3AP}
	 x+y-2z=0 ,\qquad x,y,z\in A.
\end{equation}
 Then we can fix $A\subset [N]\hookrightarrow\mathbb{Z}/(2N+1)\mathbb{Z}=G $, where the embedding is canonical. It is reasonable since {\color{red}(\ref{3AP})} has a non-trivial solution in $\mathbb{Z}$ iff it has a non-trivial solution in $G=\mathbb{Z}/(2N+1)\mathbb{Z}$.\par
 Following this, we define for every integer $k\geqslant 3$:
 \begin{defi}
 		$\displaystyle{r_{k}(N)=\frac{\sup_{A\subset[N], \text{ A does not contain non-trivial } k\text{-AP}}|A|}{N}}.$
 \end{defi}
Then the question above can be translated into whether $r_{k}(N)=o(1)$ as $N\to\infty$.\par 
In 1936, Erd\H{o}s and Tur\'{a}n also conjectured:
\begin{conj}\label{ET}
	 $\mathbb{P}$ contains arbitrarily long AP, where $\mathbb{P}=\{2,3,5,...\}$ is the set of primes.
\end{conj}
It immediately follows if we can prove $r_{k}(N)=o(1/\log N)$ for every $k\geqslant{3}$ since by the prime number theorem, $1/\log N\sim \pi(N)/N$, where $\pi(N)=|\mathbb{P}\cap[N]|.$

\subsection{Roth's Theorem and the Relative Version in the Primes }
The first result of $r_{k}(N)$ is due to Roth\cite{R53} for $k=3$.
\begin{thm}[Roth, 1953]\label{3.5}
		$\displaystyle{		r_{3}(N)\ll\frac{1}{\log\log N}},$ thus $r_{3}(N)=o(1).$ 
\end{thm}
Soon he generalized this result to the multidimensional version, but with certain conditions for the matrix of coefficients\cite{R54}:
\begin{thm}\label{3.6}
		Let $\mathbf{C}=(c_{\mu\nu})$ be a $l\times n$ matrix with integer entries. Suppose $A\subset\mathbb{Z^{+}}$ s.t. there is no solution $\mathbf{x}=(x_{1},...,x_{n})^{T}$ with all $x_{i}$ distinct to $\mathbf{Cx}=0$, then $\overline{d}(A)=0$, if $\mathbf{C}$ satisfies the following two conditions:\\
	(a)(Translation Invariance)$\sum\limits_{\nu=1}^{n}c_{\mu\nu}=0$;\\
	(b)(Large Rank)There exists $l$ linearly independent columns of $\mathbf{C}$ s.t. any of them excluded, the remaining $n-1$ columns can be devided into two groups among each there are $l$ linearly independent columns.
\end{thm}

Theorem \ref{3.5} immediately follows from Theorem \ref{3.6} if we take $\mathbf{C}=(1\ -2\ 1).$\\
\textbf{Remark.} (1) Condition (a) is necessary when we consider more general ambient groups, while condition (b) is not. For instance, if we consider $A\subset\mathbb{F}_{2}^{n}$, by taking elements in $A$ to be those with the first coordinate equal to 1 we know $A$ contains no non-trivial solution to the single linear equation $x+y-z=0$ however it has positive density $d(A)=1/2$;\\
(2) From condition (b) we know that $n>2l$, which unfortunately excludes the 4-AP case($\mathbf{C}=\begin{pmatrix}
	1 & -2 & 1 & 0\\
	0 & 1 & -2 & 1
\end{pmatrix})$ and longer cases.
\par 
In 1946, Behrend\cite{B46} constructed a dense set free of 3-AP, providing a quasi-polynomial type lower bound for $r_{3}(N).$
\begin{thm}\label{qp}
		$r_{3}(N)\geqslant\exp(-C(\log N)^{1/2})$ for some absolute constant $C>0$. 
	
\end{thm}
He showed it based on the idea that no three distinct points on a high dimensional sphere w.r.t. some expansion of integer into powers with controlled digits, lie on the same line.
Consider \[A_{k}=\{\sum\limits_{i=1}^{n} a_{i}(2d-1)^{i}:0\leqslant{i}<d,\sum\limits_{i=1}^{n}{a_{i}^{2}}=k\}.\]
We can deduce that $A_{k}$ is large for some $0\leqslant k\leqslant{n(d-1)^{2}}$ by the pigeonhole principle. 
It is also conjectured that lower bound in Theorem \ref{qp} is the best upper bound for $r_{3}(N).$
\par 
Here we give a table that shows the improvements on the upper bound for $r_{3}(N)$ since Roth's first theorem.
	\renewcommand\arraystretch{1.5}
\begin{table}[H]
	
	\centering
	\begin{threeparttable}
		\begin{tabular}{|c|c|c|}
			\hline
			Roth & 1953 & $1/\log\log N$ \\\hline
			Szemer\'{e}di & 1986 & $\exp(-O(\log\log N)^{1/2})$\\\hline
			Heath-Brown & 1987 & $1/(\log{N})^{c}$ for some tiny $c
			>0$ \\\hline
			Szemer\'{e}di & 1990 & $1/(\log N)^{1/4-o(1)}$ \\\hline
			Bourgain & 1999 & $1/(\log N)^{1/2-o(1)}$ \\\hline
			Sanders & 2011 & $(\log\log N)^{6}/\log N$ \\\hline
			Bloom & 2014 & $(\log\log N)^{4}/\log N$ \\\hline
			Bloom-Sisask & 2021 & {\color{red}$1/(\log N)^{1+c}$ for some tiny c>0}\\\hline
			Kelley-Meka & 2023 &	$\exp(-O(\log N)^{1/12})$\\\hline
			Bloom-Sisask & 2023 & $\exp(-O(\log N)^{1/9})$\\\hline
		\end{tabular}
		\begin{tablenotes}
			\footnotesize
			\item Tab. Improvements on the upper bound for $r_{3}(N)$
		\end{tablenotes}
	\end{threeparttable}
	
\end{table}
	Here we remark that the result of Bloom-Sisask in 2021\cite{BS21} was quite remarkable as we noted that the relative version in primes is implied by the bound better than $1/\log N$. Actually back in 2005 Green\cite{G05} proved the following:
	\begin{thm}[Green, 2005]\label{3.8}
			Suppose $A\subset\mathbb{P}$ s.t. $\overline{d}_{\mathbb{P}}(A)>0,$ then $A$ must contain a non-trivial 3AP.
	\end{thm}
	And in terms of primes here we also mention the most celebrated Green-Tao theorem\cite{GT08}: 
	\begin{thm}[Green-Tao, 2008]\label{3.9}
			$\mathbb{P}$ contains arbitrarily long AP.
	\end{thm}
	This answers Conjecture \ref{ET} in the affirmative. Based on the fact that the sum of reciprocals of the primes diverge, in the 1970s Erd\H{o}s gave another generalized conjecture:
	\begin{conj}\label{SR}
		If $A\subset\mathbb{Z^{+}}$ satisfes $\sum\limits_{n\in A}n^{-1}=+\infty$, then $A$ contains arbitrarily long AP.
	\end{conj} 
	In fact, both Theorem \ref{3.8} and the 3-AP version of Conjecture \ref{SR} are implied by the Bloom-Sisask bound $1/(\log N)^{1+c}$.
	\begin{proof}
		Suppose that $A$ does not contain contain non-trivial 3-AP, then $\overline{d}(A)\ll \dfrac{1}{(\log N)^{1+c}}$, thus $F(t):=|n\in{A}:n\leqslant{t}|\ll \dfrac{t}{(\log t)^{1+c}}$, so we have
		\[
		\sum\limits_{n\in A}\frac{1}{n}=\int_{1-\epsilon}^{\infty}\frac{\text{d}F(t)}{t}\leqslant 1+\int_{1-\epsilon}^{\infty}\frac{F(t)}{t^{2}}\text{d}t
\ll \int_{1-\epsilon}^{\infty}\frac{\text{d}t}{t(\log t)^{1+c} }   
	 \text{ converges.}\qedhere
		\]
	\end{proof}
\subsection{Results for $r_{k}(N)$ with $k\geqslant 4$}
We list the results for $r_{k}(N)$ with $k\geqslant 4$ in the chronological order below.
\begin{itemize}
	\item In 1975, Szemer\'{e}di\cite{S75} proved that $ r_{k}(N)=o(1)$ for every $k\geqslant 4$ using elementary combinatorial methods;
	\item In 1977, Furstenberg\cite{F77} proved the same result using ergodic methods based on the correspondence principle between combinatorial statements and the multiple reccurence theorem;
	\item In 2001 Gowers\cite{G01} introduced the higher order Fourier analysis method and gave the first quantative bound $r_{k}(N)\ll 1/(\log\log N)^{c_{k}}$ for some tiny $c_{k}>0$ for every $k\geqslant 4$;
	\item In 2005 Green-Tao\cite{GT05} proved $r_{4}(N)\ll \exp (-(\log\log N)^{c})$ for some $c>0$ and improved this result to $1/(\log N)^{c}$ in 2017\cite{GT17};
	\item In 2024, Leng-Sah-Sawhney\cite{LSS} proved $r_{k}(N)\ll \exp (-(\log\log N)^{c_{k}})$ for some $c_{k}>0$ for every $k\geqslant 5.$
\end{itemize}\par
Up to now, we haven't reached the bound stronger than $1/\log N$ for any $k\geqslant 4$ which could imply the relative version of Szemer\'{e}di's theorem in the primes and Conjecture \ref{SR}.

\subsection{Results in the Model Setting}

The integer case is based on the cyclic group setting $\mathbb{Z}/N\mathbb{Z}$. Naturally we may investigate our problem in other extreme settings of finite Abelian group, say the vector space $\mathbb{F}_{q}^{n}$ or $\mathbb{F}_{q}[t]$ over $\mathbb{F}_{q}$ with characteristic $p\geqslant 3$, which is, technically easier to deal with because they are equipped with more additive structure(e.g. subgroups), and provides a model setting for $\mathbb{Z}/N\mathbb{Z}.$\par 
	For any fixed finite Abelian group $G$, denote by $R_{3}(G)$ the size of largest 3-AP-free subset $A\subset G.$
	\begin{itemize}
		\item In 1982, Brown and Buhler\cite{BB82} showed that $R_{3}(G)=o(|G|)$ for $G$ with odd order;
		\item In 1995, Meshulam\cite{M95} showed that $R_{3}(G)\leqslant 2|G|/rk(G)$ for $G$ with odd order following Roth's argument, especially $R_{3}(\mathbb{F}_{3}^{n}) \ll 3^{n}/n$;
		\item In 2004, Lev\cite{L04} first proved for the general case that $R_{3}(G)\leqslant 2|G|/rk(2G)$;
		\item In 2012, Bateman and Katz\cite{BK12} showed that $R_{3}(\mathbb{F}_{3}^{n}) \ll 3^{n}/n^{1+c}$ for some $c>0$;
		\item In 2016, Croot-Lev-Pach\cite{CLP} first introduced the polynomial method to this problem and derived that $R_{3}((\mathbb{Z}/4\mathbb{Z})^{n}) \ll 4^{\gamma{n}}$ with $\gamma\approx 0.926$;
		\item In 2016, Ellenberg and Gijswijt\cite{EG} followed CLP's polynomial method to show that $R_{3}(\mathbb{F}_{q}^{n})\ll c^{n}$ for some $c<q$ with char $\mathbb{F}_{q}=p\geqslant 3.$
	\end{itemize}
\subsection{Other Roth-Type Problems}
We can further generalize the traditional theorems of Roth and Szemer{\'e}di, roughly speaking, in the following two ways to obtain more expected results. We only show these types of generalization and omit the quantative bounds already obtained, on which readers may refer to the mentioned papers below.\par
\textit{Generalizing both coefficients and the number of variables.}
	\begin{itemize}
	\item$A\subset \mathbb{Z}/N\mathbb{Z}$ free of non-trivial solutions to $\sum\limits_{i=1}^{s} c_{i}x_{i}=0$, where $s\geqslant 3,c_{i}\neq 0,\sum\limits_{i=1}^{s}c_{i}=0$;\cite{B12}
	\item $A\subset G$ free of non-trivial solutions to $T_{1}x_{1}+T_{2}x_{2}+T_{3}x_{3}=0$, where $T_{i}\in \text{Aut}(G),T_{1}+T_{2}+T_{3}=0$. (Thus the privious coefficient $c_{i}$ can be identified with the scalar automorphism $c_{i}\text{Id}_{G}$.)\\
	An interesting corollary of \cite{PC22}, in which a bound of $O(1/(\log N)^{1+c})$ is obtained, is the conjecture of Shkredov and Solyevery: $A\subset\mathbb{Z}^{2}$ with $\sum_{a\in A\backslash\{0\}}1/\|a\|^{2}=\infty$ contains three vertices of an isosceles right triangle.
	
\end{itemize}\par
\textit{Generalizing the equation.}
\begin{itemize}
	\item $A\subset \mathbb{Z}/N\mathbb{Z}$ free of non-trivial solutions to $\sum\limits_{i=1}^{s} c_{i}x_{i}^{d}=0$, where $s\geqslant 3,c_{i}\neq 0,\sum\limits_{i=1}^{s} c_{i}=0$(linear equation restricted in the $d$-th powers);\cite{BP17}\cite{Ch}
	\item
	(First formulated by Bergelson-Leibman\cite{BL})$A\subset \mathbb{Z}/N\mathbb{Z}$ free of polynomial progression $\{x,x+P_{1}(y),...,x+P_{m}(y)\}$ where $y\neq 0$ and $P_{i}\in \mathbb{Z}[y]$ with zero constant coefficient, for instance, $\{x,x+y,x+y^{2}\}$.\cite{PS20}\cite{PP21}
	\item $A\subset (\mathbb{Z}/N\mathbb{Z})^k$ free of certain pattern, for instance, $k=2$,  $(x,y),(x+d,y),(x,y+d^{2})$ ($d\neq 0$). \cite{PPS24}
\end{itemize}

 \section{\textsc{Fundamental Ideas in the Proofs}}\label{4}
 In this section we will explain respectively in three different settings the ideas in the proofs of main Roth's theorems we care about. Readers should bear in mind that Fourier tools remain the core of studies about the integer and relative cases, even when modern approaches include more objects from additive combinatorics such as the Bohr set and many variants of Gowers norm. However in the finite field case, which is a model setting as we mentioned, the polynomial method has its distinctive advantages.
 
 \subsection{The Integer Case}
 
\subsubsection{\textit{Global 3-AP Counting}}

	Fix $A\subset\mathbb{Z}$, denote by $\Lambda(A)$ the number of 3-APs in $A$. We have
\begin{align}
	\Lambda(A)&=\sum\limits_{a,b,c\in A} 1_{a+b=2c}=\sum\limits_{x,y\in\mathbb{Z}}1_{A}(y)1_{A}(x-y)1_{2A}(x)=\langle 1_{A}\ast 1_{A}, 1_{2A}\rangle_{\mathbb{Z}} \label{(2)}\\
	&\xlongequal{\text{Parseval}}\langle \widehat{1_{A}}^{2}, \widehat{1_{2A}}\rangle_{\mathbb{T}}=\int_{0}^{1}S(\theta)^{2}S(-2\theta)\text{d}\theta, \label{(3)}
\end{align}

where $S(\theta)=\displaystyle{\sum\limits_{n\in A}e(n\theta)}.$ \par
We remark again that in the two representatives of $\Lambda(A)$, in analytic number theory it is standard to use {\color{red}(\ref{(3)})} since the exponential sum $S(\theta)$ is involved. While to apply more combinatorial approaches we instead choose to fix $A\subset \mathbb{Z}/N\mathbb{Z}=G$ and rewrite {\color{red}(\ref{(2)})} as
\begin{equation}\label{(4)}
	\Lambda(A)=\langle 1_{A}\ast 1_{A}, 1_{2A}\rangle_{G}=\langle \widehat{1_{A}}^{2}, \widehat{1_{2A}}\rangle_{\widehat{G}}, 
\end{equation}
where we tend to use the inner product in the physical space to exploit the additive structure of $A$, though in practice inequalities are established more easily in the Fourier space. Since we may count 3-APs in some subset $V\subset G$ where $V$ shares some additive information of $A$, the inner product w.r.t. $\mu_{V}$ is thus needed.
If $A$ contains only trivial 3-AP, $\Lambda(A)=|A|$. We want to deduce from it that $|A|=\sup\limits_{\theta\in[0,1]} |S(\theta)|=|\widehat{1_{A}}|_{\infty}$ is small. An alternative more natural idea is to assume at first that $A$ is a dense set of density $\alpha$, i.e. $|A|\asymp \alpha|G|$. Now we need to bound from below either the integral {\color{red}(\ref{(3)})} or the inner product {\color{red}(\ref{(4)})}.
\subsubsection{\textit{Hardy-Littlewood Circle Method}}
	To estimate the integral $\displaystyle{\int_{0}^{1}S(\theta)^{2}S(-2\theta)\text{d}\theta}$, we shall extract the place where $S(\theta)$ is large.\\
According to Weyl's equidistribution theory, $S(\theta)$ is small when $\theta$ is close to rational points, motivating us to decompose $\mathbb{T}$(equivalently, $[0,1]$) into the major arc $\mathfrak{M}$ and the minor arc $\mathfrak{m}$, where
\[
\mathfrak{M}=\bigcup\limits_{j=1}^{k}[\theta_{j}-\frac{1}{2N},\theta_{j}+\frac{1}{2N}],\quad \mathfrak{m}=[0,1]\backslash\mathfrak{M}
\]
for some chosen rationals $\theta_{j}\in[0,1]$ and $N$ large enough s.t. points in $[\theta_{j}-\frac{1}{2N},\theta_{j}+\frac{1}{2N}]$ can be approximated by rationals with small denominators and $\displaystyle{\int_{\mathfrak{M}}|S(\theta)|^2\text{d}\theta}$ is bounded above. Whereas $\mathfrak{m}$ is small to annihilate the large value of $S(\theta)$ on it. This is the basic principle of the famous Hardy-Littlewood circle method.
\par
In proving Theorem \ref{3.5}, Roth chose the the simplest major arc $\mathfrak{M}=[\eta,1-\eta] $ for some $0<\eta<\frac{1}{2}$. This leads to a coarse bound with two logs while easily extends to Theorem \ref{3.6}, where he chose the major arc to be away from the $l$ dimensional hyperplanes defined by the column vectors in the coefficient matrix.\par 
In \cite{HB87}, Heath-Brown remarkbly created a new large sieve method, which is seen as the origin of the density increment strategy afterwards. We state this lemma below. 
\begin{lmm}
	Put $S(\theta)=\displaystyle{\sum\limits_{n=1}^{N}a_{n}e(n\theta)}(a_{n}\in \mathbb{C})$. For any $M\leqslant N$ define $\xi(M)$ by
	\[
	M\xi(M)=\max{\sum\limits_{n\in\mathcal{P}\cap[N]}|a_{n}|},
	\]
	where $\mathcal{P}$ runs over all APs of length $M$. Let $\theta_{1},...\theta_{k}\in\mathbb{R}$ satisfy $\|\theta_{i}-\theta_{j}\|_{\mathbb{R}/\mathbb{Z}}\geqslant \dfrac{1}{N},\forall i\neq j$. Then we have
	\[
	\sum_{j=1}^{k}|S(\theta_{j})|^{2}\ll N^{2}\xi(M)\xi(N),
	\]
	where $M=\lfloor \frac{1}{2}M^{1/k}\rfloor$.
\end{lmm}
By combining the new lemma, the common large sieve method and the Hardy-Littlewood circle method he reached the polylogarithmic bound. This is probably the limit of the classical analytic approach, while the density increment on AP is a heuristic start for this new strategy.
\subsubsection{\textit{Density Increment Strategy}}
The disadvantage of pure analytic method is its ignorance of additive structure of the set. If we want to get the bound $r_{3}(N)\ll 1/(\log N)^{C}$ for $C\geqslant 1/2$, say, we must dig out more information from combinatorial side.
After restricting to a finite Abelian group $G$, we introduce the \textit{Bohr set} in $G$ to be\\
\[
B=\text{Bohr}_{\nu}(\Gamma):=\{x\in G: |1-\gamma(x)|\leqslant\nu,\forall\gamma\in\Gamma\},
\]
where $\nu\geqslant 0$ is the \textit{width} of $B$ and $\Gamma\subset\widehat{G}$ is the \textit{frequency set} of $B$, rk$(B):=|\Gamma|$. This is an analogy for the subspace in $\mathbb{F}_{q}^{n}$(rank$\iff$codim), sharing good properties of strucure and iteration. Intuitively we can estimate the size of a Bohr set from its rank and width, stated in Lemma \ref{4.2}(see \cite[Lemma 4.20]{AC} for a proof).  Moreover, indeed a Bohr set contains a long AP. For a proof of Lemma \ref{4.3} readers may refer to \cite[Lemma 29]{BS23}.
\begin{lmm}\label{4.2}
If $B\subset\mathbb{Z}/N\mathbb{Z}$ is a Bohr set with rank $d$ and width $\rho\in(0,1]$ then
\[
|B|\geqslant(\rho/8)^{d}N.
\]

\end{lmm}
\begin{lmm}\label{4.3}
	Suppose $N$ is an odd prime and $B\subset \mathbb{Z}/N\mathbb{Z} $ is a Bohr set with rank $d$, then $B$ contains an AP of length $\gg |B|^{1/d}.$
\end{lmm}

	A dichotomy point of view is that an additive set behaves either random or structured. We formulate this (informally) as follows.\par
Suppose $A\subset B$ for some Bohr set $B$ with rank $d$, $d_{B}(A)=\alpha$, then either:
\par
\textit{1. (Random-Many 3-APs)} $\Lambda(A)/|B|^{2}=\langle 1_{A}\ast 1_{A}, 1_{2A}\rangle_{\mu_{B}}\gg_{d} \alpha^{3}$;\par
\textit{2. (Structured-Density Increment)} $\exists$ some Bohr set $B'$ with rk$(B')\leqslant d+\Delta(\alpha)$, $|B'|\gg_{d}\delta(\alpha)|B|$ for some $\Delta(\alpha)>0,0<\delta(\alpha)<1$ and $x\in B$ s.t. $d_{B'}((A-x)\cap B')\geqslant (1+c)\alpha$ for some $c>0.$

Fix $B^{(0)}=G, A_{0}=A, \alpha_{0}=d_{G}(A)=\alpha$. Assume at first that $\Lambda(A)$ has a discrepancy with the expected ``random" value, which forces the second case. By iterating these steps we get a maximal sequence of Bohr sets $\{B^{(k)},\text{rk}(B^{(k)})=d_{k}\}$ on which some translates of $A$, denoted by $A_{k}$, have strictly increasing densities $\alpha_{k}=(1+c)^{k}\alpha$. Maximality holds due to a trivial bound of the density. Suppose we stop with
\begin{gather}
	(1+c)^{n}\alpha\leqslant 1\text{ for some }n\in\mathbb{Z^{+}}\Longrightarrow n\ll \log(2/\alpha),\notag\\
	\text{thus }d_{n}\leqslant n\Delta(\alpha)\ll \log(2/\alpha)\Delta(\alpha).\label{(5)}
\end{gather}

By maximality of $n$, for $B^{(n)}$ we must enter the first case. Thus
\[
\Lambda(A_{n})\gg_{d_{n}}\alpha_{n}^{3}|B^{(n)}|^{2}\gg_{d_{n}}\alpha^{3}\delta(\alpha)^{2n}|G|^{2}\gg_{d_{n}}\alpha^{3}\delta(\alpha)^{\log (2/\alpha)}|G|^{2}.
\] 
Combining this with {\color{red}(\ref{(5)})} to further bound  with $\alpha$ the implicit constant w.r.t. $d_{n}$, we derive a lower bound for $\Lambda(A_{n})$, thus for $\Lambda(A)$ by translation invariance. It turns out that this strategy is rather effective and contributes to the polylogarithmic bound with large exponent and even the quasi-polynomial bound, due to Bourgain\cite{B99}, Kelley-Meka\cite{KM} and Bloom-Sisask\cite{BS21}\cite{BS23}.\par 
Now we give more comments. Firstly we will explain what randomness tells us and the motivation for the random-structured dichotomy. A naive point of view is to look at the 3-AP counting in the first case above. To find out the solution to the equation $x-2y+z=0$, we can just pick arbitrary $x,y\in A$ and then count the number of $z=2y-x$ lying in $A$. $A$ is random means, to some extent, that \[
\frac{|(x,y)\in A\times A: z=2y-x\in A|}{|A|^{2}}\approx \alpha=d_{B}(A)
\]
for any fixed $x,y\in A$, thus $\Lambda(A)\approx |A|^{2}\alpha=|B|^{2}\alpha^{3}$. However if $A$ only contains trivial 3-APs, there is a discrepancy of the value of $\Lambda(A)$ from the expected one, which then forces $A$ to be structured. In fact following \cite[Section 4.3]{AC} and \cite{BS23} we can give more descriptions of the relevant phenomena.\par 
Given an additive set $A\subset G$ where $G$ is finite, let us define first the additive energy \[
E(A,A):=|(a_{1},a_{2},a_{3},a_{4})\in A\times A\times A\times A:a_{1}+a_{3}=a_{2}+a_{4}|
\]
and the pseudo-random
\[
 \|A\|_{u}:=\sup\limits_{\xi\neq 0}|\widehat{\mu_{A}}(\xi)|.
\]
Note again that $\mu_{A}=\frac{|G|}{|A|}1_{A}$ is the function induced by the probability measure on $A$. We can actually interpret that $A$ is uniform or random if it has large sumset or small additive energy, i.e. few outputs of the sum of two elements from $A$ coincide, and a little bit more subtle, if it has small pseudo-random. Let us explain it below.\par 
In \cite{KM} Kelley-Meka discovered that we can always find a large subgroup $V\leqslant G$(in particular, a subspace of $\mathbb{F}_{q}^{n}$) on which $A$ is uniform, in the sense that $|(A+A)\cap V|$ is large, or equivalently
\begin{equation}\label{(6)}
\mu_{A}\ast\mu_{A}(x)\approx\|\mu_{A}\ast\mu_{A}\|_{1(\mu_{V})}=\frac{|G|}{|A|^{2}|V|}\sum\limits_{x\in V}\sum\limits_{a,b\in A}1_{a+b=x}=\frac{|G|}{|V|}
\end{equation}
for any $x\in V$, which can be deduced from an upper bound for  $\|\mu_{A}\ast\mu_{A}-\frac{|G|}{|V|}\|_{p(\mu_{V})}$. Motivated by this, all their work then starts with a simple H\H{o}lder lifting from the discrepancy condition\cite[Lemma 11]{BS23}:
\begin{lmm}\label{4.4}
	If $|\langle \mu_{A}\ast \mu_{A}, \mu_{C}\rangle-1|\geqslant \epsilon$ for some $\epsilon>0$, then $\|\mu_{A}\circ \mu_{A}-1\|_{p}\geqslant \frac{\epsilon}{2}$ for some $p\ll \log(2/\gamma)$, where $\gamma=d_{G}(C)$.
\end{lmm}
\begin{proof}
	By H\H{o}lder's inequality, \[
	\epsilon\leqslant|\langle \mu_{A}\ast \mu_{A}-1, \mu_{C}\rangle|\leqslant\|\mu_{A}\ast \mu_{A}-1\|_{p}\|\mu_{C}\|_{p/p-1}\leqslant 2\|\mu_{A}\ast \mu_{A}-1\|_{p}
	\]
	for sufficiently large $p\ll \log(2/\gamma)$ s.t. $\|\mu_{C}\|_{p/p-1}=\gamma^{-1/p}\leqslant 2.$\par
	Now it suffices to prove $\|\mu_{A}\ast \mu_{A}-1\|_{p}\leqslant\|\mu_{A}\circ \mu_{A}-1\|_{p}.$ In fact when $p$ is an even integer,
	 \[|\mu_{A}\ast \mu_{A}-1\|_{p}^{p}=(\widehat{\mu_{A}}^{2}1_{\neq 0_{\widehat{G}}})^{(p)}(0_{\widehat{G}})\leqslant(|\widehat{\mu_{A}}|^{2}1_{\neq 0_{\widehat{G}}})^{(p)}(0_{\widehat{G}})=\|\mu_{A}\circ \mu_{A}-1\|_{p}^{p}\]
	 where $\cdot^{ (p)    }$ is the $p$-fold convolution.\qedhere
	\end{proof}
	Back to our remark about the uniformity, the property {\color{red}(\ref{(6)})} actually correlates with the pseudo-random. A version of \cite[Lemma 4.13]{AC} tells us
	\begin{lmm}[Small pseudo-random implies uniformity/large sumsets]\label{4.5}
		For $n\in\mathbb{Z^{+}}, n\geqslant 3$ we have\[
		\alpha^{n-1}\|\mu_{A}^{(n)}-1\|_{\infty}\leqslant \|A\|_{u}^{n-2},
		\]
		where $\alpha=d_{G}(A).$\\
		In particular if $\|A\|_{u}^{n-2}<\alpha^{n-1}$, then $nA=G.$
	\end{lmm}
	Secondly, we shall note that some of the results including the formal density increment lemma and Lemma \ref{4.4} could only be applied to the finite field setting in a clean way. This is again because $\mathbb{F}_{q}^{n}$ has very nice structures $\mathbb{Z}/N\mathbb{Z}$ doesn't have. Actually Kelley-Meka only derived the uniformity lemma for $\mathbb{F}_{q}^{n}$, and chose a more ad-hoc way to deal with $\mathbb{Z}/N\mathbb{Z}$. Also, when considering density increment in $\mathbb{Z}/N\mathbb{Z}$ we should introduce the regular Bohr set which guaranteens $B+B_{\epsilon}\approx B$ where $B_{\epsilon}=$Bohr$_{\epsilon\nu}(\Gamma)$ is a narrowed copy of $B$. (Addition is not necessarily closed in general Bohr sets!)
	\begin{defi}[Regularity]
		A Bohr set of rank $d$ is regular if for all $0\leqslant\delta\leqslant 1/100d$ we have 
		\[
		(1-100d\delta)|B|\leqslant|B_{1\pm\delta}|\leqslant(1+100d\delta)|B|.
		\]
	\end{defi}
	 It is still glad to see the arguments work in the finite field setting and guide us to the integer setting. However, in the finite field setting these arguments eventually could only get the quasi-polynomial bound, which is far from the power of the polynomial method.\par
	Thirdly, as we mentioned in section \ref{1}, almost all strongest arguments on Roth-type theorems with single linear equation include the results based on the Bogolyubov-Ruzsa lemma, which is in fact an application of the almost periodicity argument by Croot-Sisask\cite{CS10}. One of many versions of almost periodcity with our familiar settings is as follows\cite[Theorem 5.4]{SS16}.
	\begin{thm}[$L^{\infty}$-almost-periodicity with Bohr sets]\label{4.7}
		Let $\epsilon\in(0,1)$. Let A, M, L be subsets of a finite Abelian group $G$, and let $B\subset G$ be a regular Bohr set of rank $d$ and width $\rho$. Suppose $|A+T|\leqslant K|A|$ for a subset $T\subset B$ with $d_{B}(T)=\tau>0$, and assume $\mu:=|M|/|L|\leqslant 1$. Then there is a regular Bohr set $B'\subset B$ with rank at most $d+d'$ and width at least $\rho\epsilon\mu^{1/2}/d^{2}d'$, where
	\[
	d'\ll \epsilon^{-2}\log^{2}(2/\epsilon\mu)\log(2/\mu)\log(2K)+\log(1/\tau),
	\]
		such that
		\[
		\|\mu_{A}\ast\mu_{M}\ast 1_{L}(\cdot+t)-\mu_{A}\ast\mu_{M}\ast 1_{L}\|_{\infty}\leqslant \epsilon \text{ for all }t\in B'.
		\]
		In particular,
		\[
		\|\mu_{A}\ast\mu_{M}\ast 1_{L}\ast\mu_{B'}-\mu_{A}\ast\mu_{M}\ast 1_{L}\|_{\infty}\leqslant \epsilon .
		\]
	\end{thm}
	From the statement we are essentially finding the ``almost periods" of $\mu_{A}\ast\mu_{M}\ast1_{L}$, which is basically saying that if $A$ is additively structured, we can always find a large Bohr set in an iterated sum or difference of sets. In philosophy it obeys the dichotomy of the density increment we have discussed, contributing to the success of Bloom-Sisask.\par
	Besides, we can sense in the introduction of additive energy, along with Lemma \ref{4.5} and Theorem \ref{4.7} where multi-fold convolutions are involved, more variables can make our problems easier. The iterated sum of sets is actually a sort of smoothing, making sets behave more random, or making functions after multi-fold convolutions more periodic. Back in the 2010s, quasi-polynomial bounds have been established for single linear equation with more than four variables such as $x+y+z-3w=0$\cite{SS16}. Thus to study $r_{3}(N)$ we need to introduce more variables. Bloom-Sisask successfully combined the arguments of Kelley-Meka such as H\H{o}lder lifting and unbalancing, with Theorem \ref{4.7} via the so-called ``dependent random choice", finally deducing the integer version of the uniformity lemma\cite[Theorem 5]{BS23}. They thought the dependent random choice was the most crucial step in their whole paper, which transformed the $L^{p}$ information into some cutting information of more variables. We state the finite field version here\cite[Lemma 8]{BS23}. Note that $A_{1},A_{2},S$ play the role of $A,M,L$ in Theorem \ref{4.7} respectively.
	\begin{prop}
		If $\|\mu_{A}\circ\mu_{A}\|_{p}\geqslant 1+\epsilon$ for some $\epsilon>0$, then there are $ A_{1},A_{2}\subset A$ of density at least $\alpha^{2p+O_{\epsilon}(1)}$ s.t. 
		\[
		\langle \mu_{A_{1}}\circ \mu_{A_{2}}, 1_{S}\rangle\geqslant 1-\frac{\epsilon}{4},
		\]
		where $S=\{x:\mu_{A}\circ\mu_{A}>1+\frac{\epsilon}{2}\}.$
	\end{prop}
	Lastly, combining \cite[Theorem 5]{BS23} with Lemma \ref{4.3}, there is an important by-product(\cite[Theorem 3]{BS23}) stated as follows.
	\begin{thm}
		If $A\subset [N]$ has size $\alpha N$, then $A+A+A$ contains an AP of length\[\geqslant\exp(-C\log(2/\alpha)^{3})N^{c/\log(2/\alpha)^{9}},\]
		where $C,c>0$ are constants.
	\end{thm}
 \subsection{Transference Principle in the Prime Case}
 
 In Green's paper\cite{G05} he created the influential transference principle to study the Roth's theorem in the primes. This was also applicable for other relative version of Roth's theorem and intepreted as the dense model theorem\cite[Section 9.4]{GTAC}. Briefly speaking, if we want to study the relative version on $S\subset \mathbb{Z}/N\mathbb{Z}$, we can find some "dense model" $B$ of $A\subset S$ in the sense that
 \begin{equation}
 	\frac{|B|}{N}\approx \frac{|A|}{|S|}=d_{S}(A).
 \end{equation} 
 From the Fourier prospective it is exactly to find an approximation
 \begin{equation}\label{(8)}
  \widehat{1_{B}}\approx\widehat{1_{A}},
 \end{equation}
 thus by {\color{red}(\ref{(3)})} we have 
 \[
 \Lambda(A)=\langle \widehat{1_{A}}^{2}, \widehat{1_{2A}}\rangle\approx\langle \widehat{1_{B}}^{2}, \widehat{1_{2B}}\rangle=\Lambda(B).
 \]
 So we may apply the Roth's theorem in the integers to obtain a bound for $\Lambda(A)$, as long as the error in the approximation is bounded. \par 
 However the approximatin as stated in {\color{red}(\ref{(8)})} is too much to hope for, since we are requesting\[
 |S|\geqslant|A|=\widehat{1_{A}}(0)\approx \widehat{1_{B}}(0)=|B|.
 \]
 This means $S$ itself should be a dense set, which is unpleasant to see. Take $S=\mathbb{P}_{N}:=\mathbb{P}\cap[N]$ for example, from $|B|\lessapprox|S|$ we are basically requesting the result $r_{3}(N)\ll 1/\log N$ from the transference principle, which is appearently harder before any result of Roth's theorem in the primes.\par 
 To get around this, we need to weight the function $1_{A}$ with some ``majorant" $\nu$, and establish the approximation $\widehat{1_{B}}\approx \widehat{f}$ for $f=1_{A}\nu$. Here we give some definitions.
 \begin{defi}[Majorant]
 	Given $S\subset[N]$, we define a majorant on $S$ to be a non-negative function $\nu:\mathbb{Z}\to[0,\infty)$ with $\text{supp}(\nu)\subset S$ s.t.\[
 	\sum\limits_{n}\nu(n)=(1+o(1))N.
 	\]
 \end{defi}
 \begin{defi}[Fourier decay]
 	We say a majorant $\nu$ on $[N]$ has Fourier decay of level $\theta$ if
 	\[
 	\|\widehat{\nu}-\widehat{1_{[N]}}\|_{\infty}\leqslant \theta N.
 	\]
 \end{defi}
 \begin{defi}[Restriction at $p$]
 		We say a majorant $\nu$ on $[N]$ satisfies a restriction estimate at exponent $p$ if
 		\[
 		\sup\limits_{|\phi|\leqslant\nu}\int_{\mathbb{T}}|\widehat{\phi}(\alpha)|^{p}\text{d}\alpha\ll_{p}\|\nu\|_{1}^{p}N^{-1}.
 		\]
 \end{defi}
 Green\cite{G05} essentially showed the following.
 \begin{thm}[Green, 2005]
 	Suppose that the majorant $\nu$ has Fourier decay of level
 	$\theta$ and satisfies a restriction estimate at exponent $p$. Then for any $0\leqslant f\leqslant \nu$ there exists
 	$0\leqslant g\ll 1_{[N]}$ s.t.
 	\[
 	\|\widehat{f}-\widehat{g}\|_{\infty}\ll_{p}\log (\theta^{-1})^{-\frac{1}{p+2}}N.
 	\]
 \end{thm}
 In terms of $S=\mathbb{P}_{N}$ he chose a majorant taking values some propotion of $\log{p}$ with prime $p$ restricted on some AP to a large modulus, say $m=\prod_{p\leqslant W}p$ for some $W\asymp \log\log N$. This is the so-called ``W-trick", making the primes more uniform in the sense that $\sup\limits_{\xi\neq 0}|\widehat{\nu}(\xi)|$ is small, which aims to control $g$ so that it is really ``set-like". Then he chose $g=f\ast\beta\ast\beta$, where $\beta=1_{B}/|B|$ for some Bohr set $B\subset\mathbb{Z}/N\mathbb{Z}$ with frequency set being the large spectrum
 \[
 \Gamma=\text{Spec}(f,\delta):=\{r:\arrowvert\widehat{f}(r)\arrowvert\geqslant\delta\}
 \]
 for some $\delta>0$. Here through the canonical isomorphism between dual groups $B$ can be written as \[B=\{x:\left\|\frac{xr}{N}\right\|_{\mathbb{R}/\mathbb{Z}}\leqslant\epsilon,\forall r\in\Gamma\}\]
 for some $\epsilon\in(0,1)$.\par
 It should be noted that the large spectrum is a nice candidate for the frequency set, since we can easily control a function by extracting the large Fourier coefficients. Bohr sets thus defined also share nice additive properties(see \cite[Proposition 4.39, 4.40]{AC}, for example).\par 
 As for the restriction estimate, Green built the following\cite[Theorem 1.5]{G05}.
 \begin{thm}[Hardy-Littlewood majorant property]\label{4.14}
 	Suppose $p\geqslant 2$ is a real number, $\{a_{n}\}_{n\in\mathbb{P}_{N}}$ be a sequence of complex numbers with $|a_{n}|\leqslant 1$ for all $n$. Then
 	\[
 	\left\|\sum\limits_{n\in\mathbb{P}_{N}}a_{n}e(n\theta)\right\|_{L^{p}(\mathbb{T})}\ll_{p}\left\|\sum\limits_{n\in\mathbb{P}_{N}}e(n\theta)\right\|_{L^{p}(\mathbb{T})}.
 	\]
 \end{thm}
 He pointed out that this property holds not only for primes. In fact the general version of this implies some important Fourier restriction conjectures and thus relates to Kakeya problems, too. A generalization of Theorem \ref{4.14} is \cite[Theorem 1.12]{LD}, which covers the Piatetski–Shapiro primes.\par 
 We shall also remark that due to the application of W-trick the resulted estimate would be poor. This seems to be the main flaw of transference principle when we need to uniformize the weighted characteristic function.
 \subsection{Polynomial Method in the Finite Field Case}
 	In 2009, Dvir\cite{kkf} first introduced the polynomial method to prove the analogy of Kakeya conjecture on finite fields.
 \begin{thm}[Dvir, 2009]
 	Suppose $K\subset\mathbb{F}_{q}^{n}$ is a Kakeya set, i.e. $\forall x\in\mathbb{F}_{q}^{n},\exists y\in\mathbb{F}_{q}^{n}$ s.t. $L_{y,x}:=\{y+a\cdot x:a\in\mathbb{F}_{q}\}\subset K$. Then $|K| \gg_{n} q^{n}.$
 \end{thm}
 He actually proved that $|K|\geqslant\binom{q+n-1}{n}=|M_{n}|$, where $M_{n}=\{\text{monomials in }\mathbb{F}_{q}[x_{1},...,x_{n}], \deg\leqslant q-1\}.$ For any field $\mathbb{F}$, by viewing polynomials of $n$ variables over $\mathbb{F}$ as functions on $\mathbb{F}^{n}$, we know the following fundamental fact.\\
\textbf{Fact. }Suppose $S\subset\mathbb{F}^{n}$ s.t. $|S|<\dim{V}$, where $V$ is a linear subspace of polynomials with $ \deg\leqslant d$ over $\mathbb{F}$, then $\exists g\in V\,\backslash\{0\}$ s.t. $g$ vanishes on $S$.

 Here we can use the fact by taking  $\mathbb{F}=\mathbb{F}_{q},	S=K, V=S_{n}=\text{Span}_{\mathbb{F}_{q}}M_{n}$, then $|K|<|M_{n}|=\dim{S_{n}}$ forces a $0\neq g\in S_{n}$ vanishing on $K$, thus vanishing on every $L_{y,x}$, from which we can deduce a contradiction that $g=0.$\par
 	Combining the fact above and the following lemma, it is enough to prove $R_{3}(\mathbb{F}_{q}^{n})\ll c^{n}$ for some $c<q.$
 \begin{lmm}$\mathrm{\cite[Proposition\ 2]{EG}}$
 	Define $M_{d}:=\{\text{monomials in }\mathbb{F}_{q}[x_{1},...,x_{n}], \deg\leqslant d$ and each variable has $\deg\leqslant q-1\},S_{d}:=\text{Span}_{\mathbb{F}_{q}}M_{d},m_{d}:=|M_{d}|=\dim S_{d}.$\par
 	Suppose $A\subset\mathbb{F}_{q}^{n},\alpha,\beta,\gamma\in\mathbb{F}_{q}$ s.t. $\alpha+\beta+\gamma=0.$ $P\in S_{d}$  s.t. $P(\alpha a+\beta b)=0$ for $\forall a,b\in A,a\neq b.$ Then $|c\in A: P(-\gamma c)\neq 0|\leqslant 2m_{d/2}.$
 \end{lmm}
 The solution-free condition can be intepreted as $S(A)\cap (-\gamma A)=\phi$, where $S(A)=\{\alpha a+\beta b:a,b\in A,a\neq b\}.$ It now suffices to consider $V=\{P\in S_{d}:P\text{ vanishes on }\mathbb{F}_{q}^{n}\backslash(-\gamma A)\}$, $S$ be a maximal support of any $P\in V$ and apply the fact above.
 \section{\textsc{Prospective Work}}
 In this section the author will list two of his prospective work to do in Roth-type problems. Although few breakthroughs of these two problems have occured in recent years, both of their results are very weak compared to many similar problems. The author believes the tools on hand is sufficient for improvement even if it is small, and he is eager for any new ideas and would appreciate any suggestions.
\subsection{Relax the Conditions in Multidimensional Linear Roth's theorem}\label{5.1}
 In Theorem \ref{3.6}, condition (b) requries the coefficient matrix of our concern to have the large rank, which is obviously too strict for the truth.  \par 
 As we remarked, condition (b) excludes the case $k$-AP for all $k\geqslant 4$ with the corresponding matrix satisfying $l=n-2$, only allowed when $n=3$ in Theorem \ref{3.6}. Even in 1975 Szemer\'{e}di\cite{S75} solved $k$-AP for all $k$ the only improvement to Theorem \ref{3.6} since 1954\cite{R54}, as the author know, is still due to Roth\cite{R70} and \cite{R72} in 1970 and 1972 repectively. In those two papers relaxations to the rank condition were small, only including $n=2l$, thus allowing 4-AP. Few attention has been attracted to this version of Roth's theorem since then, probably because on the one hand, different types of linear system vary a lot in both difficulty and methodology. On the other, only certain types are of interest and we do not need the most general result. \par 
 Nevertheless, the author conjectures that we can totally remove condition (b) in Theorem \ref{3.6}. Although we cannot expect any good quantative bound, in view of that the best bound for $r_{k}(N)(k\geqslant 5)$ currently is $\exp(-(\log\log N)^{c_{k}})$, we will be glad to see how much we can relax the condition such that $\overline{d}(A)=0$ for $A$ free of non-trivial solution to a linear system.\par 
 The author has tried to test the adaptability of density increment strategy, which seems to be the most promising approach to linear Roth's theorems now. However there is a density gap here. Following the notation of Theorem \ref{3.6}, to count the number of solutions we need to first set our ambient group as $G=(\mathbb{Z}/N\mathbb{Z})^{l}$, then denote by $\Delta\leqslant G$ the diagonal subgroup of $G$, i.e., with $l$ coordinates equal. More generally we will define the diagonal set of $X\subset \mathbb{Z}/N\mathbb{Z}$ to be 
 \[
 \Delta(X):=\{(x,...,x):x\in X\}\subset \Delta.
 \]
Thus the linear system is transformed into a single linear equation
 \begin{equation}\label{(9)}
 \mathbf{c}_{1}\cdot\mathbf{x}_{1}+\cdots \mathbf{c}_{n}\cdot\mathbf{x}_{n}=0,
 \end{equation}
where $ \mathbf{c}_{i}\in G\backslash\{0\}$ while $\mathbf{x}_{i}\in \Delta$ for all $1\leqslant i\leqslant n$, and the dot multiplication is the common inner product by coordinate. The number of solutions in $A$ to the original system then equals the number of solutions in $\Delta(A)$ to {\color{red}(\ref{(9)})}, which is 
\begin{equation}\label{(10)}
\Lambda(A)=\langle1_{\mathbf{c}_{1}\cdot\Delta(A)}\ast\cdots\ast 1_{\mathbf{c}_{n-1}\cdot\Delta(A)},1_{\mathbf{c}_{n}\cdot \Delta(A)}\rangle,
\end{equation}
where the convolution and the inner product are both taken globally. In \cite{R54} Roth just transformed {\color{red}(\ref{(10)})} into the Fourier space and estimated the integral of exponential sums as usual. However to apply the density increment strategy, we need to tackle with the diagonal sets dilated with different coefficient $\mathbf{c}_{i}$, forcing us to consider the relative problem in $\Delta$ with a large density gap $d_{G}(\Delta)=|G|^{-\frac{l-1}{l}}$. Note that we are unable to start the density increment in $\Delta$ since any global version of related crucial lemmas usually ask for translation invariance, while $\Delta$ certainly does not satisfy since any translation of $\mathbf{x}\in G\backslash\Delta$ will break this. Of course it is also difficult to find some dense model since the distribution of $\Delta$ is too imbalanced. It even intersects with an AP at more than one point only if the AP is contained in itself. To get around it we have to either find some other approaches designed for the diagonal sets, or deduce from the global information {\color{red}(\ref{(10)})} more effective local information of $A$.
 \subsection{Improve Roth's Theorem in the Squares}
 In non-linear equations, we especailly care about the Roth-type problems in quadratic equations. The homogeneous one is
 \begin{equation}\label{(11)}
 \sum\limits_{i=1}^{s} c_{i}x_{i}^{2}=0.
 \end{equation}
 In higher degrees we usually need more variables, to be specific in \cite{Ch} Chow proved for general homogeneous equation in degree $d$ and $s\geqslant d+1$ with a four-log bound. While the best known result of equation {\color{red}(\ref{(11)})} is still due to Browning and Prendiville\cite{BP17} in 2017, with $s\geqslant 5$ and a three-log bound. Despite the same scale of density loss as in \ref{5.1}, transference principle is well applied here.\par 
 The fatal difference between non-linear equations and linear ones is the translation invarinace. In \cite{H15} Henriot proved that a bound of $1/(\log N)^{c}$ can be derived by consider the system
 \[
  \sum\limits_{i=1}^{s} c_{i}x_{i}^{2}=0, \sum\limits_{i=1}^{s} c_{i}x_{i}=0.
 \]
 Thus he could combine the transference principle with density increment strategy to get around the poor estimate caused by the W-trick. To achieve the same effect in {\color{red}(\ref{(11)})}, there should be more structural results of $A^{2}$ for some $A\subset[N]$, say finding a long AP in $A^{2}-A^{2}$. If there is, the large gap between the current three-log bound and the true magnitude could be bridged soon. 
	\bibliography{books}
	\vspace{2ex}
	
	\textsc{School of Mathematical Sciences, University of Science and Technology of China, Hefei 230026, China}\par
	\textit{Email Address}: \href{mailto:llcwwzhang@mail.ustc.edu.cn}{\color{black}\texttt{llcwwzhang@mail.ustc.edu.cn}}
\end{document}